\documentclass[12pt,a4paper]{amsart}
 \pdfoutput=1
 
 \usepackage{amsmath,amssymb,amsthm}
 \usepackage{bm}  
 \usepackage{mathtools}
\usepackage{tikz}
\usetikzlibrary{arrows}
 \usepackage{thmtools}
 \usepackage{float} 
 
 \usepackage{xcolor} 	
 \usepackage{hyperref}
 \hypersetup{
 	colorlinks,
     linkcolor={red!60!black},
     citecolor={green!60!black},
     urlcolor={blue!60!black},
 }
 \usepackage[maxbibnames=99]{biblatex}
 \usepackage[utf8]{inputenc}
 \usepackage{csquotes}
 \usepackage[T1]{fontenc}
 \usepackage{lmodern}
 \usepackage[babel]{microtype}
 \usepackage[english]{babel}
 
 \usepackage{geometry}
 \usepackage{enumitem}

\theoremstyle{plain}
\newtheorem{theorem}{Theorem}[section]
\newtheorem{fact}[theorem]{Fact}
\newtheorem{proposition}[theorem]{Proposition}

\newtheorem{claim}[theorem]{Claim}
\newtheorem{corollary}[theorem]{Corollary}
\newtheorem{lemma}[theorem]{Lemma}

\newtheorem{observation}[theorem]{Observation}

\newtheorem{conjecture}[theorem]{Conjecture}

\newtheorem*{theorem*}{Theorem}
\newtheorem*{corollary*}{Corollary}

\theoremstyle{definition}
\newtheorem{remark}[theorem]{Remark}
\newtheorem{definition}[theorem]{Definition}

\title{Finite matchability under the matroidal Hall's condition}

\author{Attila Jo\'{o}}
\thanks{Funded by the Deutsche Forschungsgemeinschaft (DFG, German
Research Foundation)-513023562 and partially by NKFIH OTKA-129211}
\address{Attila Jo\'{o},
Department of Mathematics, University of Hamburg, Bundesstra{\ss}e 55 (Geomatikum), 20146 Hamburg, Germany}
\email{attila.joo@uni-hamburg.de}
\address{Attila Jo\'{o},
Logic, Set theory and topology department, Alfr\'{e}d R\'{e}nyi Institute of Mathematics,  13-15 Re\'{a}ltanoda St., 
Budapest, Hungary}
\email{jooattila@renyi.hu}

\begin{document}

\begin{abstract}
Aharoni and Ziv conjectured that if $ M $ and $ N $ are finitary matroids on $ E $, then a certain ``Hall-like'' condition 
is sufficient to guarantee the existence of an $ M $-independent spanning set of 
$ N $. We show that their condition ensures that every finite subset of $ E $ is $ N $-spanned by an $ M $-independent set. 
\end{abstract}
\maketitle

\section{Introduction}
Hall's classical theorem \cite{hall1935onrep} states that a finite bipartite graph $ G=(S,T,E) $ has no matching that covers $ T $ iff there exists an  
$ 
X\subseteq T $ with $ \left|N(X)\right| < \left|X\right|$. The literal generalization of this theorem to infinite graphs fails. Indeed, in the classical 
counterexample (``playboy paradox'') 
 $ 
\left|S\right|=\left|T\right|=\aleph_0 $, there is a $ t\in T $ that is connected to each point in $ S $ and $ E(G-t) $ is a 
perfect matching of $ G-t $. On the one hand, no matching covers $ T $ because each $ t'\in T-t $ has a unique neighbour and  $ 
N(T-t) = S $. On the other hand, $ \left|N(X)\right| \geq \left|X\right|$ for every $ X\subseteq T $. The ``right'' generalization of Hall's 
theorem which is also valid for infinite graphs was discovered by Aharoni. The key idea is that the 
sizes of $ X $ and $ N(X) $ should be 
compared via injections along 
graph edges:
\begin{theorem}[{Infinite Hall's theorem by Aharoni, \cite[Theorem 3.2]{aharoni1991infinite}}]\label{thm: infinite Hall}
A bipartite graph $ G=(S,T,E) $ has a matching that covers $ T $ iff there is no $ X\subseteq T $ such that $ N(X) $ 
can be matched to $ X $ 
and no matching covers $ X $.
\end{theorem}
\noindent Theorem \ref{thm: infinite Hall} turned out to be an efficient tool, the infinite version of König's Duality Theorem was derived from it 
\cite{aharoni1984konig}, i.e. every bipartite 
graph (finite or infinite) 
has a 
matching $ F $ such that one can obtain a vertex-cover by choosing precisely one endpoint from each $ e \in F $. Other types of criteria of 
matchability were introduced by Nash-Williams \cite{nash1975marriages}. For more on infinite matching theory (including the discussion of the 
non-bipartite case) we 
refer to \cite{aharoni1991infinite}.

Bipartite matchings are special cases of common independent sets of two matroids. Let $ G=(S,T,E) $ be a bipartite graph. For $ v 
\in S\cup T $, let $ U_{\delta(v), 1} $ be the $ 1 $-uniform matroid (i.e. the singletons and the empty set are the independent sets) on the set 
$ \delta(v) $ of edges incident with $ v $. Consider the direct sums  

 \[M:= \bigoplus_{s \in S}U_{\delta(s), 1}\ \   N:= \bigoplus_{t \in T}U_{\delta(t), 1}. \]
 Then matchings of $ G $ are exactly the common independent sets of $ M $ and $ N $. Moreover, if  $ G $ has no isolated 
 vertices in $ T $, then a matching $ F $ covers $ T $ iff it is a base of $ N $. This led Aharoni and Ziv 
 to call in  \cite{aharoni1998intersection} the common independent sets of two matroids ``matroidal matchings''. Let us call the 
 pair $ (M,N) $ of matroids matchable if there exists an $ M $-independent  base of $ N $ and finitely matchable if every finite set is $ N 
 $-spanned by an $ M $-independent set. The following matroidal generalization of Hall's theorem was obtained by Edmonds (implicit in 
 \cite{edmonds2003submodular}):
 
 \begin{theorem}[Edmonds]\label{thm: matroidal matching finite}
Assume that $ M $ and $ N $ are matroids on the common finite edge set $ E $. Then $ (M,N) $ is matchable iff there is no  $ X\subseteq E $ 
with  $ r_{M}(X)<r_{N.X}(X) $.
\end{theorem}
A standard way to extend the concept of matroids to allow infinite ground set  is provided by the following axiomatization (see for example 
\cite{oxley1978infinite} or 
\cite[745]{schrijver2003combinatorial}):
\begin{definition}\label{def: finitary matroid}
A finitary matroid\footnote{Some authors call it ``infinite matroid'', ``independence space'' or ``independence structure''.} is a pair  
$M=(E,\mathcal{I})$ with ${\mathcal{I} \subseteq \mathcal{P}(E)}$ such that  
\begin{enumerate}
[label=(\Roman*)]
\item\label{item: matroid axiom1} $ \emptyset\in  \mathcal{I} $;
\item\label{item: matroid axiom2} $ \mathcal{I} $ is closed under taking subsets;
\item\label{item: matroid axiom3} If $ I,J\in \mathcal{I} $ are finite with $ \left|I\right|<\left|J\right| $, then there exists an $  e\in J\setminus I $ 
such that
$ I+e\in \mathcal{I} $;
\item\label{item: matroid axiom4} If all finite subsets of  an infinite set $ X $ are in $ \mathcal{I} $, then $ X\in \mathcal{I} $.
\end{enumerate} 
\end{definition}
\noindent Theorems \ref{thm: infinite Hall} and \ref{thm: matroidal matching finite} led Aharoni and Ziv to expect that the following statement 
holds:
\begin{conjecture}[{\cite[Conjecture 3.3]{aharoni1998intersection}}]\label{conj: Hall MIC}
Assume that $ M $ and $ N $ are finitary matroids on the common edge set $ E $. Then there is no $ M $-independent base 
of $ N $ iff there exists 
an  $ X\subseteq E $ such that $ M \!\upharpoonright\! X $ has an  $ N.X $-independent base but  $ N.X $ has no $ M 
\!\upharpoonright\! X 
$-independent base.\footnote{Technically they used a slightly stronger condition they call ``unhindered'' but our ``if and only if'' formulation of 
their conjecture is known to be
equivalent.}
\end{conjecture}
They reduced in the same paper the Matroid Intersection Conjecture of Nash-Williams \cite[Conjecture 1.2]{aharoni1998intersection} to 
Conjecture \ref{conj: Hall MIC}. In all the partial 
results towards to Conjecture \ref{conj: Hall MIC} so far, structural or size restrictions are assumed about the matroids
\cite{aharoni1998intersection, aigner2018intersection, bowler2015matroid, joo2021MIC}. In this paper, we 
provide the first step towards Conjecture \ref{conj: Hall MIC} in the original setting by proving that $ (M,N) $ is 
finitely matchable under 
the matroidal Hall's condition, i.e. under the assumption given in Conjecture \ref{conj: Hall MIC}. More precisely, we show the 
following:
 \begin{theorem}\label{t: intromain}
Assume that $ M $ and $ N $ are finitary matroids on the common edge set $ E $ and there is no
 $ X\subseteq E $ such that $ M \!\upharpoonright\! X $ has an  $ N.X $-independent base but  $ N.X $ has no $ M 
\!\upharpoonright\! X 
$-independent base. Then $ (M,N) $ is finitely matchable, i.e. every finite  $ F\subseteq E $ is $ N $-spanned by an $ M $-independent set.
\end{theorem}

 While in Conjecture \ref{conj: Hall MIC} we want to find an $ M $-independent set that $ N $-spans the whole $ E $, in this paper 
we have 
the more modest goal to find (under the same assumption) for every finite $ F\subseteq E $ an $ M $-independent set that $ N $-spans it.  Maybe 
surprisingly, this seemingly small step already requires an 
entirely new 
machinery compared to previous partial results. We believe these new tools are important for further investigations of the conjecture.

The paper is structured as follows. In the following section, Section \ref{sec: Notation and preliminaries}, we introduce the notation and recall a 
few previous results that we need later. Section \ref{sec: Main result} is devoted to the proof of the main result and it is subdivided into three 
subsections. In Subsection \ref{subsec: negli and stable}, the concept of negligible and stable sets are introduced and a couple of reductions are 
done. A preorder on the finite common independent sets of two finitary matroids is defined in Subsection \ref{subsec: equiv classes} and the 
induced equivalence classes are investigated. It is shown (Lemma \ref{lem: reachability in D(I)}) that the ``reachable by a 
directed path'' relation corresponding to Edmonds' auxiliary digraph associated with a finite common independent set depends only on the 
equivalence class. In the last part, Subsection \ref{subsec: maximal directed set}, we factorize the preorder with this equivalence relation and take 
a maximal directed set in the resulting poset. The analysis of this poset concludes the proof of the main result Theorem \ref{t: main}.\\

\section{Notation and preliminaries}\label{sec: Notation and preliminaries}
We write $ \boldsymbol{\bigcup \mathcal{F}} $ for the union of the sets in $ \mathcal{F} $. The symmetric difference $ \boldsymbol{X \triangle 
Y} $ is defined as $ (X\setminus Y)\cup (Y\setminus X) $.  A pair ${M=(E,\mathcal{I})}$ is a \emph{finitary matroid} 
if ${\mathcal{I} \subseteq \mathcal{P}(E)}$ satisfies  the axioms 
\ref{item: matroid axiom1}-\ref{item: matroid axiom4}.  We refer to $ E $ as the edge set or ground set of the matroid.
The sets in~$\mathcal{I}$ are called \emph{independent} and the sets $ \mathcal{P}(E)\setminus \mathcal{I} $ are \emph{dependent}.
The maximal independent sets are the \emph{bases} while the minimal dependent sets are the \emph{circuits} of the matroid. 
For an  ${X \subseteq E}$, ${\boldsymbol{M  \!\!\upharpoonright\!\! X} :=(X,\mathcal{I} 
\cap \mathcal{P}(X))}$ is a matroid and it  is called the \emph{restriction} of~$M$ to~$X$.
We write ${\boldsymbol{M - X}}$ for $ M  \!\!\upharpoonright\!\! (E\setminus X) $. Let $ B $ be a base of $ M  \!\!\upharpoonright\!\! X $.
The \emph{contraction} of $ X $ in $ M $ is the matroid 
$\boldsymbol{M/X}=(E\setminus X, \mathcal{I}')$ where $ I \in \mathcal{I}' $ iff $ B\cup I \in \mathcal{I} $. This is well-defined, i.e. $ 
\mathcal{I}' $ does not depend on the choice of $ B $. Let $\boldsymbol{M.X}:= M/(E\setminus X) $.   
If $ I $ is  independent in $ M $  but $ I+e $ is dependent for some $ e\in E\setminus I $,  then there is a unique 
circuit   $ \boldsymbol{C_M(e,I)} $ of $ M $ through $ e $ contained in $ I+e $.  We say~${X 
\subseteq E}$ \emph{spans}~${e \in E}$ in matroid~$M$ if either~${e \in X}$ or there exists a circuit~${C 
\ni e}$ with~${C-e \subseteq X}$. 
By letting $\boldsymbol{\mathsf{span}_{M}(X)}$ be the set of edges spanned by~$X$ in~$M$, we obtain a finitary closure 
operator\footnote{Axiom \ref{item: matroid axiom4} ensures that the closure operator $ 
\mathsf{span} $ of finitary matroids is finitary. This is in contrast to the more general matroid concept discovered independently by Higgs and 
Bruhn et al. \cite{bruhn2013axioms}.} 
$ \mathsf{span}_{M}: \mathcal{P}(E)\rightarrow \mathcal{P}(E) $. 
An ${S \subseteq E}$ is \emph{spanning} in~$M$ if~${\mathsf{span}_{M}(S) = E}$.  

 \begin{fact}\label{fact: finite subset spans}
If $ M $ is a finitary matroid and $F,X\subseteq E(M)  $ with $ F \subseteq \mathsf{span}_M(X) $ where $ F $ is finite, then there is already a 
finite $ 
X'\subseteq X $ with $ F \subseteq \mathsf{span}_M(X') $.
\end{fact}

Let an ordered pair $ 
(M,N) $ of finitary matroids defined on $ E $ be fixed. The \emph{matroidal Hall's condition} $ \boldsymbol{\mathsf{Hall}(M,N)} $ asserts that 
for every $ X\subseteq E $ for which there is an $ N.X 
$-independent base 
of $ M \!\!\upharpoonright\!\! X $, there is also an $ M \!\!\upharpoonright\!\! X $-independent base of $ N.X $. 

\begin{observation}\label{obs: Hall herdeit}
$ \mathsf{Hall}(M,N) $ implies that $ \mathsf{Hall}(M - X,N/X) $ holds for every $ X \subseteq E $. 
\end{observation}
\noindent We say that $ (M,N) $ is \emph{finitely 
matchable} if for every finite $ F\subseteq E $ there is an $ I \in \mathcal{I}_M $ with $ F \subseteq 
\mathsf{span}_N(I) $. Note that by Fact \ref{fact: finite subset spans} $ I $ can be chosen to be finite.

The classical ``augmenting path technique'' developed by Edmonds \cite{edmonds2003submodular} to prove his Matroid intersection theorem 
will play an 
important role in our proofs. For $ I \in \mathcal{I}_M \cap \mathcal{I}_N $, let 
$ \boldsymbol{D(I)}=D_{(M,N)}(I) $ be a digraph on $ E $ with the following arcs. For  $ x\in E\setminus I $ and $ y\in I $,  
$ xy\in D(I) $  if   $ x\in \mathsf{span}_M(I) $ with $ y\in C_M(x,I) $ and $ yx\in D(I) $   if   $ x\in \mathsf{span}_N(I) 
$ with 
$ y\in C_N(x,I) $.  We call these the arcs of  \emph{first type} and arcs of \emph{second type} respectively. For $ X\subseteq E $, we write $ 
\boldsymbol{E(X,I)} $ for 
the set of those elements of $ E $ 
that are reachable from $ X $ in $ D(I) $ by a directed path\footnote{By directed path we always mean forward-directed path.}. The 
set of 
arcs entering (leaving) vertex set $ X $  in digraph 
$ 
D $ is denoted by $ 
\boldsymbol{\delta_D^{-}(X)} $  ($ \boldsymbol{\delta_D^{+}(X)} $)  and we write $ \boldsymbol{N_D^{-}(X)} $ and $ 
\boldsymbol{N_D^{+}(X)} $ respectively for the in-neighbours and out-neighbours of $ X $, i.e. the tails of the arcs in $ \delta_D^{-}(X) $ and 
the heads of the arcs in $ \delta_D^{+}(X)$. 
An 
\emph{augmenting path} for $ I $  is a directed path in $ D(I) $ from $  E\setminus 
\mathsf{span}_{N}(I) $ to $  E\setminus \mathsf{span}_{M}(I) $ without ``jumping arcs''. More precisely, it is a 
set  $ P=\{ x_0,\dots, x_{2n} \}\subseteq E $ such that 
\begin{enumerate}
\item $ x_0 \in E\setminus \mathsf{span}_{N}(I) $,
\item $ x_{2n}\in E\setminus \mathsf{span}_{M}(I) $,
\item $ x_ix_{i+1}\in D(I) $ for $ i<2n $,
\item $ x_i x_j \notin D(I) $ if $  i+1<j \leq 2n $.
\end{enumerate}
\noindent An $ xy$\emph{-augmenting path} is an augmenting path starting at $ x $ and terminating at $ y $. An  
$ x $\emph{-augmenting path} is an augmenting path starting at $ x $.  The following lemma is due to Edmonds (for a proof with our notation see 
\cite[Corollary 
3.2]{joo2021MIC}).
\begin{lemma}\label{lem: aug path x_0 x_2n}
If $ P=\{ x_0,\dots, x_{2n} \} $ is an augmenting path for $ I\in \mathcal{I}_M \cap \mathcal{I}_N $, then $ I \triangle P \in 
\mathcal{I}_M \cap \mathcal{I}_N $ with $ \mathsf{span}_{M}(I \triangle P)=\mathsf{span}_{M}(I+x_{2n}) $ and 
$ \mathsf{span}_{N}(I \triangle P)=\mathsf{span}_{N}(I+x_{0})  $.
\end{lemma}
 \begin{lemma}[Edmonds, \cite{edmonds2003submodular}]\label{lem: Edmonds}
If $ E $ is finite, then  $ I\in \mathcal{I}_M \cap \mathcal{I}_N $ is a largest common independent set iff there is no augmenting path for $ I $.
\end{lemma}

\begin{lemma}[{\cite[Lemma 3.1]{joo2021MIC}}]\label{lem: arc pres aug}
If  $ P $ is an augmenting path for $ I $ and $ x \in E $ such that $ P \cap (\{ x \}\cup N^{+}_{D(I)}(x))=\emptyset $, then  $ 
N^{+}_{D(I)}(x)\subseteq 
N^{+}_{D(I \triangle P)}(x) $.
\end{lemma}
\noindent For $ I \in \mathcal{I}_M\cap \mathcal{I}_N $, let $ \boldsymbol{\mathcal{A}(I)} $ consist of those elements of $ \mathcal{I}_M\cap 
\mathcal{I}_N $ that can be obtained from $ I $ by  a finite iteration of applying augmenting paths, i.e. 
\[  \mathcal{A}(I)=\{ I \triangle P_0 \triangle \cdots \triangle P_{n-1}:\ n \in \mathbb{N} \wedge P_i \text{ is an augmenting path for } I 
\triangle P_0 \triangle \cdots \triangle P_{i-1} \}.  \]
\section{The main result}\label{sec: Main result}

\begin{theorem}\label{t: main}
If $ M $ and $ N $ are finitary matroids satisfying $ \mathsf{Hall}(M,N) $, then $ (M,N) $ is finitely matchable. 
\end{theorem}

\subsection{Reductions via negligible and stable sets}\label{subsec: negli and stable}
\begin{definition}\label{def: negli}
A set $ G\subseteq E $ is $ (M,N) $-\emph{negligible} if for all finite sets $ X \subseteq G $ and  $ Y \subseteq E\setminus G $ there 
exists an $ M/Y $-independent set which $ N $-spans $ X $.  
\end{definition}
\noindent We write simply ``negligible'' when  $ (M,N) $ is clear from the context. The empty 
set is always negligible, we call it the \emph{trivial} negligible set.
\begin{observation}\label{obs: negliDef}
In order to check the negligibility of a set, it is enough 
to consider $ X $ without $ N $-loops and $ Y $ without $ M $-loops.
\end{observation}
\begin{lemma}\label{lem: negl the negl}
If $ G $ is negligible and $ (M-G, N/G) $ is finitely matchable, then $ (M,N) $ is also finitely matchable. 
\end{lemma}
\begin{proof}
Let a finite $ F\subseteq E $ be given. Since $ (M-G, N/G) $ is finitely matchable, we can pick an $ (M-G) $-independent  finite $ 
I_0
$ that $ 
N/G $-spans $ F\setminus G $. The matroid $ N $ is finitary, thus there is a finite $ G_0\subseteq G $ such that $ I_0 \cup G_0 $ spans $ 
F\setminus 
G $ in $ N $. By the negligibility of $ G $, there is an $ M/I_0 $-independent $ I_1 $ that spans $G_0 \cup(F \cap G) $ in $ N $. Finally, $ I:=I_0 
\cup I_1 $ is independent in $ M $ and spans $ F $ in $ N $.
\end{proof}
\begin{lemma}\label{lem: negl iterateable}
If $ G $ is $ (M,N) $-negligible and $ G' $ is $ (M-G, N/G) $-negligible, then $ G \cup G' $ is $ (M,N) $-negligible.
\end{lemma}
\begin{proof}
Let $ X\subseteq G \cup G' $ and $ Y\subseteq E\setminus (G\cup G') $ be finite sets. Since $ G' $ is $ (M-G, N/G) $-negligible, there is a 
 finite $ 
I'\in \mathcal{I}_{(M-G)/Y} $  that $ N/G $-spans $ X\cap G' $. The matroid $ N $ is finitary, thus there is a finite $ G_0\subseteq G $ 
such that $ I' \cup G_0 $ spans $ X\setminus G $ in $ N $. Since $ G $ is $ (M,N) $-negligible, there is an $ 
I\in \mathcal{I}_{M/(Y\cup I')} $ that $ N $-spans $  G_0 \cup (X \cap G) $. By construction,  $ I \cup I' \in \mathcal{I}_{M/Y}$ and it spans $ X 
$ in $ 
N $.
\end{proof}
\begin{observation}\label{obs: negli zorn sat}
The union of a $ \subseteq $-chain of negligible sets is negligible.
\end{observation}
\begin{proof}
Let $ \mathcal{L} $ be a chain of negligible sets and let finite sets $ X \subseteq \bigcup \mathcal{L} $ and $ Y \subseteq E\setminus (\bigcup 
\mathcal{L}) $ be given. Since $ X $ is finite, there is a single $ G \in \mathcal{L} $ with $ X\subseteq G $. The set $ G $ is negligible, thus there 
is an $ M/Y $-independent $ I $ that spans $ X $ in $ N $ and we are done.
\end{proof}

\begin{corollary}\label{cor: max negl}
There exists a maximal negligible set.
\end{corollary}
 \begin{proof}
Observation \ref{obs: negli zorn sat} ensures that Zorn's lemma is applicable.
\end{proof}
\begin{claim}\label{clm: reduction 1}
Theorem \ref{t: main} is implied by its restriction to pairs $ (M,N) $ without non-trivial negligible sets.
\end{claim}
\begin{proof}
Let $ (M,N) $ be given. By Corollary \ref{cor: max negl} there exists a maximal $ (M,N) $-negligible set $ G $. On the one hand, 
 $ \mathsf{Hall}(M-G, N/G) $ holds by Observation  \ref{obs: Hall herdeit}. On the other hand, Lemma \ref{lem: negl iterateable} ensures 
 that there is no 
 non-trivial $ (M-G, N/G) $-negligible set.  By knowing the theorem for such pairs, we may conclude that $ (M-G, N/G) $ is finitely matchable. 
 But then so is $ (M,N) $ by Lemma \ref{lem: negl the negl}.
\end{proof}

\begin{definition}\label{def: stable}
A set $ I\in\mathcal{I}_M\cap \mathcal{I}_N  $ is called $ (M,N) $\emph{-stable} if  $  \mathcal{I}_{M/I}\cap \mathcal{I}_N\subseteq  
\mathcal{I}_{N/I}$.   
\end{definition}
\noindent In other words, $ I $ is stable if  whenever we take a common independent set $ J $ that is disjoint from $ I $, then the $ M 
$-independence of $ I \cup J $ implies its $ N $-independence. We say simply ``stable'' if $ (M,N) $ is clear from the context. 
\begin{remark}
The name ``stable'' was motivated by the fact that by applying augmenting paths iteratively starting with a stable $ I $, none of the augmenting 
paths 
meet $ I $. However, we are not going to use this observation directly.
\end{remark}
The empty set obviously is stable.
\begin{observation}\label{obs: union of stable}
The union of a $ \subseteq $-chain of stable sets is stable.
\end{observation}
\begin{proof}
Let $ \mathcal{L} $ be a chain of stable sets with $ I:=\bigcup\mathcal{L} $ and suppose that $ J \in  \mathcal{I}_{M/I}\cap \mathcal{I}_N$. If $ 
J $ is dependent in $ N/I $, then there already is some $ I' \in \mathcal{L} $ such that $ J $ is $ N/I' $-dependent because  of Fact 
\ref{fact: finite subset spans}. Since 
$  \mathcal{I}_{M/I}\subseteq  \mathcal{I}_{M/I'} $, this contradicts the stability of $ I' $.
\end{proof}
Stable sets can be ``merged'' in some sense even if they do not form a $ \subseteq $-chain:
\begin{lemma}\label{lem: merge stable}
For every set $ \mathcal{F} $ of stable sets there exists a single stable set $ I\subseteq \bigcup \mathcal{F} $ with $ \bigcup 
\mathcal{F}\subseteq  \mathsf{span}_M(I) $.
\end{lemma}
\begin{proof}
Fix an enumeration $ \mathcal{F}=\{ I_\alpha:\ \alpha<\kappa \} $. We build an increasing continuous sequence $ (J_\alpha)_{\alpha\leq\kappa} $ 
of stable sets such that $ J_\alpha $ spans $ \bigcup_{\beta<\alpha}I_\beta $ in $ M $. Clearly, $ J_0:=\emptyset $ is suitable. 
Stability is preserved at limit steps by Observation \ref{obs: union of stable}. Suppose that $ J_\alpha $ is already defined for some $ 
\alpha<\kappa 
$. Let $ J_{\alpha+1} $ be an extension of $ J_\alpha $ to a maximal $ M $-independent subset of $ J_\alpha \cup I_\alpha $. Then  $ J_{\alpha+1} 
$ is $ M $-independent and spans $ \bigcup_{\beta<\alpha+1}I_\beta $ in $ M $ by construction.  Then  $ 
J_{\alpha+1}\setminus J_\alpha $ is $ M/J_\alpha $-independent, moreover, it is also $ N $-independent because so is $ I_\alpha 
\supseteq  
J_{\alpha+1}\setminus J_\alpha $. Thus the stability of  $ J_\alpha $ ensures that  $ J_{\alpha+1}\setminus J_\alpha $ is $ N/J_\alpha 
$-independent and therefore $ J_{\alpha+1} $ is $ N 
$-independent. 

Now we show that $ J_{\alpha+1} $ is stable. Let a $ J \in  
\mathcal{I}_{M/J_{\alpha+1}}\cap \mathcal{I}_N$ be given. Then $J \in \mathcal{I}_{M/I_\alpha} $ as well because $ I_\alpha \subseteq 
\mathsf{span}_M(J_{\alpha+1}) $ by construction. Thus the stability of $ I_\alpha $  implies that $ J $ is $ N/I_\alpha $-independent. But then $ J 
$ is $ N/(J_{\alpha+1}\setminus J_\alpha) $-independent by $ J_{\alpha+1}\setminus J_\alpha \subseteq I_\alpha $. As we have already noticed, 
 $ J_{\alpha+1}\setminus J_\alpha  $ is $ N $-independent. Therefore the 
$ N $-independence of $ J 
\cup (J_{\alpha+1}\setminus J_\alpha) $ follows. But then the stability of 
$ J_\alpha $ implies that $ J_\alpha \cup [J \cup (J_{\alpha+1}\setminus J_\alpha)]=J \cup J_{\alpha+1}$ is $ N $-independent. Thus $ J $ is $ 
N/J_{\alpha+1} $-independent.
\end{proof}
\begin{corollary}\label{cor: largest stable}
There exists a stable set that $ M $-spans all stable sets.
\end{corollary}
 \begin{proof}
We apply Lemma \ref{lem: merge stable} with the set $ \mathcal{F} $ of all stable sets.
\end{proof}

\begin{lemma}\label{lem: stable redaction}
If there is no non-trivial $ (M,N) $-negligible set, then there is an $ M $-spanning set $ S $ such that for every $ x\in S $ there is an 
$(M\!\!\upharpoonright\!\! 
S, N \!\!\upharpoonright\!\! S) $-stable set that spans $ x $ in $ M $.
\end{lemma}

\begin{claim}\label{clm: reduction 2}
Lemma \ref{lem: stable redaction} implies Theorem \ref{t: main}.
\end{claim}
\begin{proof}
Let $ (M,N) $ be given.  By Claim \ref{clm: reduction 1} we may assume that there is no non-trivial $ (M,N) $-negligible set. Let $ S $ be as in 
Lemma \ref{lem: stable redaction}.  By applying 
Corollary \ref{cor: largest stable}, we pick an $(M\!\!\upharpoonright\!\! S, N\!\!\upharpoonright\!\! S) $-stable set $ I $ 
that $ M \!\!\upharpoonright\!\! 
S 
$-spans all $(M\!\!\upharpoonright\!\! S, N \!\!\upharpoonright\!\! S) $-stable sets. 
By Lemma \ref{lem: stable redaction}, $ I $ is 
spanning in $ M \!\!\upharpoonright\!\! S $ and hence in $ M $ as 
well.  Obviously $ I  \in 
\mathcal{I}_M\cap \mathcal{I}_N $ because $ I \in 
\mathcal{I}_{M \upharpoonright S}\cap \mathcal{I}_{N \upharpoonright S} $ by definition. Thus $ I $ is 
an $ N $-independent base of $ M $. But then $ \mathsf{Hall}(M,N) $ ensures the existence of an $ M $-independent base of $ N $. In 
particular, $ (M,N) $ is finitely matchable (and therefore $ E=\emptyset $, since otherwise $ E $ would be a non-trivial negligible set).
\end{proof}

It remains to prove Lemma \ref{lem: stable redaction}. To do so, we need some tools.
\subsection{Equivalence classes of finite common independent sets}\label{subsec: equiv classes} 

Let $ \mathsf{Fin}(I_M\cap I_N)  $ be the set of the finite elements of $ I_M\cap I_N $. For $  I, J \in \mathsf{Fin}(I_M\cap I_N)  $, we let $ I 
\trianglelefteq J $ if $ I \subseteq 
\mathsf{span}_M(J)\cap \mathsf{span}_N(J) $. Clearly, $ \trianglelefteq $ is a preorder. Let $ I \sim J $ if $ I\trianglelefteq J 
$ and $ I \trianglerighteq J $. This is an equivalence relation on $ \mathsf{Fin}(I_M\cap I_N)  
$ and $ \trianglelefteq $ induces a partial order on $ \mathsf{Fin}(I_M\cap I_N) /\sim 
$ in the usual way. We abuse the notation and 
denote this partial order also by $ \trianglelefteq $. The following two lemmas provide a characterisation of the equivalence 
classes and a way to transition between the elements of a fixed class.
\begin{lemma}\label{l: switch cycle}
If $ O\subseteq E $ induces a chordless directed cycle in $ D(I) $ for some $ I\in \mathcal{I}_M \cap \mathcal{I}_N  $, then $ I \triangle O\in 
\mathcal{I}_M \cap \mathcal{I}_N  $ with $ I \sim  I \triangle O$. 
\end{lemma}
\begin{proof}
By symmetry, it is enough to show that $ I \triangle O\in \mathcal{I}_M $ with $ \mathsf{span}_M(I)=\mathsf{span}_M(I 
\triangle O) $. Let $ O=\{ x_0,\dots, x_{2n+1} \} $ where  $ x_ix_j\in D(I) $ iff $ j=i+1 $ 
(mod $ 2n+2 $) and $ x_{2n+1}\in 
I $.  For $ i\leq n+1 $ we let  \[ I_i:=(I\setminus \{ x_{2j+1}:\ n\geq j\geq n+1-i \})\cup \{ x_{2j}:\ n\geq  j\geq n+1-i \}.  \]
Note that $ I_0=I $ and $ I_{n+1}=I \triangle O $. We show by induction on $ i $ that $ I_{i}\in \mathcal{I}_M $, $ 
\mathsf{span}_M(I_i)=\mathsf{span}_M(I) $ and $ C_M(x_{2k}, I)=C_M(x_{2k}, I_i)$ for every $ 0\leq k\leq n-i $. For $ i=0 $ this is a 
tautology 
and 
for $ i=n+1 $ this is exactly what we want to prove. Suppose we know the statement for some $ i\leq n $. Then  for $ k=n-i $ we have
$ C_M(x_{2(n-i)}, I)=C_M(x_{2(n-i)}, I_i)$ and hence  $ x_{2(n-i)+1}\in C_M(x_{2(n-i)}, I_i) $. It follows that  $ 
I_{i}-x_{2(n-i)+1}+x_{2(n-i)}\in \mathcal{I}_M $  and $ I_{i}-x_{2(n-i)+1}+x_{2(n-i)} $ has the same $ M $-span as $ I_i $. But $ 
I_{i}-x_{2(n-i)+1}=I_{i+1} $ by definition, therefore $ \mathsf{span}_M(I_{i+1})=\mathsf{span}_M(I_i)=\mathsf{span}_M(I) $. Finally, $ 
C_M(x_{2k}, I)=C_M(x_{2k}, I_{i+1})$ for  $ 0\leq k\leq n-1-i $ because $ C_M(x_{2k}, I)=C_M(x_{2k}, I_{i}) $ for $0\leq k\leq n-1-i $ 
by 
induction and $ x_{2(n-i)+1}\notin C_M(x_{2k}, I) $ for $ k\leq n-1-i $ because $ O $ is chordless by assumption.

\end{proof}
We call an $ O $ as above a \emph{switching cycle} (w.r.t. $ I $ and $ (M, N) $). 
\begin{lemma}\label{l: char equi}
 $ I \sim J $ iff there are $  I_0,\dots, I_n\in 
\mathsf{Fin}(I_M\cap I_N) $ with $ I_0=I $ and $ I_n=J $ such that 
$ I_{i+1}=I_i \triangle 
O_i $ for a suitable switching cycle $ O_i $ for $ I_i $ for every $ i<n $. 
\end{lemma}
\begin{proof}
The ``if'' direction follows from Lemma \ref{l: switch cycle}. We prove the ``only if'' part by induction on $ \left|I \triangle J \right| $.
If $ \left|I \triangle J \right| =0$, then $ I=J $ and there is nothing to prove.
Assume $ I\neq J $. It is enough to find a switching cycle $ O\subseteq I \triangle J $ for $ I $ because after that we are done by applying the 
induction hypothesis with $ I \triangle O $ and $ J $. Let $ M' $ 
and $ N' $ be the finite matroids that we obtain from $ M $ and $ N $ respectively by contracting 
$ I\cap J $ and 
deleting $ E\setminus (I\cup J) $.  The sets $ I\setminus J $ and $ J\setminus I $ are common bases of $ M' $ and $ N' $ and partition their 
common ground set. It follows that in the finite digraph $ D_{(M', N')}(I\setminus J) $ there are neither sources nor sinks hence it cannot 
be acyclic. Finally, any 
chordless directed cycle $ O $ in $ D_{(M', N')}(I\setminus J) $ is such a cycle in $ D_{(M,N)}(I) $ with $ O\subseteq I \triangle J $ which 
concludes the proof.
\end{proof}

\begin{lemma}\label{lem: reachability in D(I)}
If $ I\sim J $, and  $ y$ is reachable by a directed path from $ x $ in $ D(I) $, then $ y $ is reachable from $ x $ in $ D(J) $ 
as well.
\end{lemma}
\begin{proof}
 By Lemma \ref{l: char equi}, we can assume without loss of generality that $ J=I \triangle O $ where $ O $ is a switching cycle for $ 
I $.  It is enough to show that directed cuts are preserved, i.e. if  $ \delta_{D(I)}^{+}(F)=\emptyset $ for an $ 
F\subseteq E $, then  $ 
\delta_{D(I\triangle O)}^{+}(F)=\emptyset $ as well. Recall that each arc of $ D(I) $ either enters $ I $ or leaves $ I $ and we call these 
 first 
and second type of arcs respectively. We must have either 
$ 
O\subseteq F $ or $ O \subseteq E\setminus F $ because $ O $ induces a directed cycle in $ D(I) $. Suppose first that  $ O\subseteq F $.  Since $I 
\sim I\triangle O $, we have
\begin{align}
\mathsf{span}_M(I\triangle O) \cap F&=\mathsf{span}_M(I)\cap F \label{eq: switch}\\
\mathsf{span}_N(I\triangle O) \setminus F&=\mathsf{span}_N(I)\setminus F. \label{eq: switch 1}
\end{align}
\noindent By the 
definition of $ D(I) $, the assumption  ``no arc of the first type leaves $ F $'' means that 
\begin{equation}\label{eq: no 1st arc leaves}
\mathsf{span}_M(I) \cap F=\mathsf{span}_M(I\cap F) \cap F.
\end{equation}
Similarly, the assumption that ``no arc of the second type leaves $ F $'' means that 
\begin{equation}\label{eq: no 2st arc leaves}
\mathsf{span}_N(I) \setminus F=\mathsf{span}_N(I\setminus F) \setminus F.
\end{equation}
It follows from (\ref{eq: no 1st arc leaves})  and $ O\subseteq F $ that $ O $ is  a 
switching cycle w.r.t. $ F\cap I $ and $ {(M \!\!\upharpoonright\!\! F, N/(I\setminus F) \!\!\upharpoonright\!\! F)} $ as well. 
Thus by Lemma \ref{l: switch cycle}:
\begin{equation}\label{eq: same span}
\mathsf{span}_M(I \cap F)  =\mathsf{span}_M((I\cap F) \triangle O)=\mathsf{span}_M((I\triangle O)\cap F). 
\end{equation}
\noindent Therefore  
 \[ \mathsf{span}_M(I\triangle O) \cap F\overset{(\ref{eq: switch})}{=}\mathsf{span}_M(I)\cap F\overset{(\ref{eq: no 1st arc leaves})}{=} 
 \mathsf{span}_M(I\cap F)\cap 
 F\overset{(\ref{eq: same span})}{=}\mathsf{span}_M((I\triangle O)\cap F) \cap F. \] This means that 
 no arc of the first type leaves $ F $ in $ D(I\triangle O) $. Similarly
 \[ \mathsf{span}_N(I \triangle O) \setminus F \overset{(\ref{eq: switch 1})}{=}  \mathsf{span}_N(I)\setminus F 
 \overset{(\ref{eq: no 2st arc leaves})}{=}\mathsf{span}_N(I\setminus F)\setminus F =\mathsf{span}_N((I \triangle O)\setminus F) \setminus F, \]
where the last equality follows from the fact that $ I\setminus F=(I \triangle O)\setminus F $ because $ O\subseteq F $. This means that 
 no arc of the second type leaves $ F $ in $ D(I\triangle O) $. Therefore $ 
 \delta_{D(I\triangle O)}^{+}(F)=\emptyset $.

 The case $ O \subseteq 
 E\setminus F $ can be reduced to the case $O\subseteq F $ by considering 
 $E\setminus F $ and $ (N,M) $ instead of $ F $ and $ (M,N) $ since $ D_{(N,M)}(I) $ is obtained from $ D_{(M,N)}(I) $ by reversing all 
 the arcs. 
\end{proof}

\begin{corollary}\label{cor: aug [I]}
Assume that $ I\sim J $. Then there is an $  xy $-augmenting path  for $ I $ iff there is an $ xy$-augmenting path for $ J$. 
Furthermore, if $ P $ is an $ xy $-augmenting path for $ I $ and $ Q $ is an $ xy $-augmenting path for $ J $, then 
$ I \triangle P \sim J \triangle Q $.
\end{corollary}
 \begin{proof}
The first part follows directly from Lemma \ref{lem: reachability in D(I)}. The second part follows from the fact
that by using Lemma \ref{lem: aug path x_0 x_2n} and $ I \sim J $ we have:
 \begin{align*}
\mathsf{span}_M(I \triangle P)&= 
\mathsf{span}_M(I+y)=\mathsf{span}_M(\mathsf{span}_M(I)+y)\\&=\mathsf{span}_M(\mathsf{span}_M(J)+y)=\mathsf{span}_M(J+y)=
\mathsf{span}_M(J \triangle Q)\text{ and}\\
\mathsf{span}_N(I \triangle P)&= 
\mathsf{span}_N(I+x)=\mathsf{span}_N(\mathsf{span}_N(I)+x)\\&=\mathsf{span}_N(\mathsf{span}_N(J)+x)=\mathsf{span}_N(J+x)=
\mathsf{span}_N(J \triangle Q).
 \end{align*}
\end{proof}
 We say that $ [J] $ is the $xy $\emph{-augmentation} of   $ [I] $ if there is 
 an $ xy $-augmenting path $ P $ for $ I $ with $ J\sim I \triangle P $. Similarly,  $ [J] $ is an  $ x 
 $\emph{-augmentation} of $ [I] $ if there is 
  an $ x $-augmenting path $ P $ for $ I $ with $ J\sim I \triangle P $.   If there is no such 
 $ [J] $, then the $ xy $-augmentation ($ x $-augmentation) of $ [I] 
 $ does not exist. Corollary \ref{cor: aug [I]} ensures that these are well-defined. Let $ \boldsymbol{\mathcal{A}([I]) }$ be the set of those $ [J] $ 
 that can be obtained by a finite iteration of augmentations from 
 $ [I] $.

\begin{lemma}\label{lem: reach some J'}
If $ [J] \trianglerighteq[I] $, then  $  [J]\cap \mathcal{A}(I) \neq \emptyset $.
\end{lemma}
\begin{proof}
We apply augmenting paths w.r.t. the finite matroids $ (M\!\!\upharpoonright\!\! (I \cup J),N\!\!\upharpoonright\!\!(I \cup J)) $ 
iteratively starting with $ I $ and doing 
as 
many iterations as possible.
Note that these are augmenting paths for $ (M,N) $ as well. For the resulting $ J' $, we have $ J'\trianglelefteq J $ because $ J $ is a 
common base of  $ M \!\!\upharpoonright\!\! (I \cup J) $ and $ N\!\!\upharpoonright\!\! (I \cup J) $.  On the other hand,  $ J' $ is a maximal-sized 
common 
independent set  of the finite matroids $ M \!\!\upharpoonright\!\! (I \cup J) $ and $ 
N\!\!\upharpoonright\!\! (I \cup J) $ by  (see Lemma \ref{lem: Edmonds}). 
Therefore  $ J' $ is also a 
common base, 
thus $ J' \sim J $ and hence $ J' \in [J]\cap \mathcal{A}(I)$.
\end{proof}
\begin{corollary}\label{cor: reach by aug}
$ [J] \trianglerighteq [I] $ iff $ [J]\in \mathcal{A}([I]) $.
\end{corollary}

\subsection{A maximal directed set}\label{subsec: maximal directed set}
Let $\boldsymbol{\mathcal{D}}$ be a maximal $ \trianglelefteq 
$-directed subset of $\mathsf{Fin}(I_M\cap I_N) /\!\!\sim $ (exists by Zorn's lemma). 

\begin{observation}\label{obs: nonempty dowclosed D}
 $ \mathcal{D} $ is a non-empty, downward closed subset of $\mathsf{Fin}(I_M\cap I_N) /\!\!\sim $.
\end{observation}
\begin{proof}
$ \mathcal{D}\cup \{  [\emptyset]  \} \cup \{ [I]:(\exists [J]\in \mathcal{D})([I]\trianglelefteq [J])\} $ is a directed set and cannot be a proper 
superset of $ 
\mathcal{D} $ by the maximality. 
\end{proof}
We let
\begin{align*}
\boldsymbol{\mathsf{span}_M(\mathcal{D})}&:=\bigcup_{[I]\in \mathcal{D}} \mathsf{span}_{M}(I) \\
\boldsymbol{\mathsf{span}_N(\mathcal{D})}&:=\bigcup_{[I]\in \mathcal{D}} \mathsf{span}_{N}(I).
\end{align*}
For $ [I]\in \mathcal{D} $ we define $ 
\boldsymbol{{[I]}{\uparrow_{\mathcal{D}}}}:=\{ [J]\in 
\mathcal{D}:  
[J]\trianglerighteq [I]\} $.

\begin{lemma}\label{l: [I_e]}
For every $ x\in  E\setminus \mathsf{span}_N(\mathcal{D})  $, there is an $ I_x\in \bigcup\mathcal{D} $ 
such that there is no $ [J]\in{[I_x]}{\uparrow_{\mathcal{D}}} $  for which there exists an $ x
$-augmentation  of $ [J] $.
\end{lemma}
\begin{proof}
 Let $ x\in E\setminus \mathsf{span}_N(\mathcal{D})   $ be given. For $ y\in E $, we define $ \mathcal{D}_{x,y} $ to be the set of all  $ [I] \in  
  \mathcal{D}$ admitting  an $ xy $-augmentation. 
  \begin{observation}\label{obs: xy aug no cofin}
  There is no $ y\in E 
    $ such that $ \mathcal{D}_{x,y} $ is cofinal in $ (\mathcal{D}, \trianglelefteq) $.
  \end{observation}
  \begin{proof}
  Suppose for a contradiction that $ \mathcal{D}_{x,y} $ is cofinal. Then $ \mathcal{D}_{x,y}\neq \emptyset $ because $ 
  \mathcal{D}\neq \emptyset $  by Observation \ref{obs: nonempty dowclosed D}.  On the one hand, none of the $ 
    xy $-augmentations of the elements of $ \mathcal{D}_{x,y} $ are in $ \mathcal{D} $ because of the choice of $ x $. On the other hand, adding 
    these to $ \mathcal{D} $ results in a directed set contradicting the maximality of $ \mathcal{D} $.
  \end{proof}
    
 \begin{observation}\label{obs: xM aug no cofin}
 $\mathcal{D}_{x,M}:=\cup \{\mathcal{D}_{x,y}:\ y\in \mathsf{span}_M(\mathcal{D}) \} $ is not cofinal in $ \mathcal{D} $.
 \end{observation}
 \begin{proof}
 As in the previous observation, if it were cofinal, then adding the corresponding augmentations to $ \mathcal{D} $ would yield a directed set 
 which is a proper superset of $ \mathcal{D} $.
 \end{proof}
By Observation \ref{obs: xM aug no cofin}, we can pick an $ I_x\in \bigcup\mathcal{D} $ such that $ 
{[I_x]}{\uparrow_{\mathcal{D}}} \cap 
\mathcal{D}_{x,M}=\emptyset $. Suppose for a contradiction that there is a $ [J]\in{[I_x]}{\uparrow_{\mathcal{D}}} $  
that has an $ xy 
$-augmentation for some $ 
y $. Let $ P $ be an $ xy 
$-augmenting path for $ J $. We show that $ \mathcal{D}_{x,y} $ is cofinal in $ \mathcal{D} $ which contradicts Observation 
\ref{obs: xy aug no cofin}. To do so, it is 
enough to prove that $ \mathcal{D}_{x,y}\supseteq{[J]}{\uparrow_{\mathcal{D}}} $ because $ \mathcal{D} $ is a directed 
set and therefore $ 
{[J]}{\uparrow_{\mathcal{D}}} $ 
is cofinal in $ \mathcal{D} $.  Since $ {[J]}{\uparrow_{\mathcal{D}}}=\mathcal{A}[J]\cap \mathcal{D} $ by Corollary \ref{cor: 
reach by aug}, it 
suffices to show that $ P $ is an 
augmenting path for every $ K\in \mathcal{A}(J) \cap \mathcal{D} $. This holds 
for $ K=J$ by assumption. Suppose that we already know for  some $ K\in \mathcal{A}(J) \cap \mathcal{D} $ that $ P $ is an 
augmenting path for $ K $ and assume that $ Q $ is an 
augmenting 
path for $ K $ for which $ [K \triangle Q]\in \mathcal{D}  $.  Note that by Lemma 
\ref{lem: aug path x_0 x_2n} the endpoint of $ Q $ is in $ 
\mathsf{span}_M(\mathcal{D}) $.   Then $ Q $ cannot contain any element that is reachable 
from $ x $ in $ D(K) $ since otherwise 
there would be a directed path in $ 
D(K) $ from $ x $ to the endpoint of $ Q $ contradicting $ {[I_x]}{\uparrow_{\mathcal{D}}} \cap 
\mathcal{D}_{x,M}=\emptyset $. But then $ P \cap Q\neq \emptyset $, moreover, $ Q\cap N^{+}_{D(K)}(x) $ for each $ x\in P $. Thus Lemma 
\ref{lem: arc pres aug} ensures that $ P $ remains a directed path in $ D(K 
\triangle Q) $ 
because the arcs witnessing that $ P $ form a directed path in $ D(K) $ remain arcs in $ D(K\triangle Q) $ as well. 
This concludes the proof of Lemma \ref{l: [I_e]}.
\end{proof}
Let us fix sets $ \boldsymbol{I_x} $ for $ x \in E\setminus \mathsf{span}_N(\mathcal{D})  $ according to Lemma \ref{l: [I_e]}. Note that $ 
E(x,I)= E(x,J)$ if $ I \sim J $  
by 
Lemma \ref{lem: reachability in D(I)}. 
\begin{lemma}\label{lem: E(x, I) increasing}
Assume that $[J], [I]\in {[I_x]}{\uparrow_{\mathcal{D}}} $ with $ J \in \mathcal{A}(I) $. Then $ E(x,I)\subseteq E(x, J) $ 
and $ E(x,I) \cap I=E(x,I)\cap J $. 
\end{lemma}
\begin{proof}
It is enough to prove the special case where there is a single augmenting path $ P $ for $ I $ such that $ J=I \triangle P $. Indeed, the general case 
follows then by induction on the number of augmenting paths transforming $ I $ to $ J $. By the definition of $ I_x $ we know that $ P \cap 
E(x,I)=\emptyset $ and therefore $ E(x,I) \cap I=E(x,I)\cap J $.  Lemma \ref{lem: arc pres aug} ensures that all the arcs of $ D(I) $ whose both 
endpoints are in $ E(x,I) $ remain arcs in $ D(J) $. Thus the reachability of a $ y\in E(x,I) $  in $ D(J) $ from $ x $ is witnessed by the same 
directed path as in $ D(I) $. Thus $ E(x,I)\subseteq E(x, J) $.
\end{proof}

\begin{corollary}\label{cor: rachability in D([I])}
If $[J], [I]\in {[I_x]}{\uparrow_{\mathcal{D}}} $ with $ [J] \trianglerighteq [I] $, then $ E(x, J)\supseteq E(x,I)  $. 
\end{corollary}
\begin{proof}
Recall that $ [J] \trianglerighteq [I] $ implies $ [J]  \in \mathcal{A}([I])$ by Corollary \ref{cor: reach by aug} and 
$ E(x, J)=E(x, J') $ whenever $ J' \sim J $ by Lemma \ref{lem:  reachability in D(I)}. We pick a $ J' \sim J $ with $ J'\in 
\mathcal{A}(I) $ by using 
Lemma \ref{lem: reach some J'} and then apply Lemma \ref{lem: E(x, I) increasing} with $ J' $ and $ I $.
\end{proof}
\begin{lemma}\label{lem: E(x,I) M-span}
For   $ [I] \in {[I_x]}{\uparrow_{\mathcal{D}}} $ we have $ 
E(x,I)\subseteq \mathsf{span}_M(I 
\cap E(x,I)) $.
\end{lemma}
\begin{proof}
Let $ y \in E(x,I)\setminus I $ be given. Then $ y \in \mathsf{span}_M(I) $ since otherwise the $ xy $-augmentation of $ [I] $ would exist which 
contradicts $ [I] \in {[I_x]}{\uparrow_{\mathcal{D}}} $. But then either $ y $ is an $ M $-loop, or $ C_M(y,I) $ is 
well-defined and a subset of $ E(x,I) $ because 
$ yz 
\in D(I) $ for each $ z \in C_M(y,I)-y $. In both cases $ y \in \mathsf{span}_M(I \cap E(x,I)) $. 
\end{proof}
We define \[ \boldsymbol{W}:=\bigcup\{E(x,I): x\in E\setminus \mathsf{span}_N(\mathcal{D})  \wedge [I] \in 
{[I_x]}{\uparrow_{\mathcal{D}}} \}. \]
 Note that $ E\setminus \mathsf{span}_N(\mathcal{D}) \subseteq W $ (see Figure \ref{fig: JX and JY}).  We let $ 
\boldsymbol{S}:= W \cap \mathsf{span}_N(\mathcal{D}) $, this will be the $ S $ promised in Lemma \ref{lem: stable redaction}.  By 
the definition of $ S $, we can fix for each $ y 
\in S $ an  $ \boldsymbol{x_y}\in W\setminus 
\mathsf{span}_N(\mathcal{D}) $ and a $ \boldsymbol{J_y}\in \bigcup {[I_{x_y}]}{\uparrow_{\mathcal{D}}} $ such that  
$ y\in 
E(x_y, J_y) $.

\begin{observation}\label{obs: reach from y}
$ E(y, J)\subseteq W $ whenever $ y \in S $ and $ [J] \in {[J_y]}{\uparrow_{\mathcal{D}}} $.
\end{observation}
\begin{proof}
The set  $ E(x_y, J) $ appears in the union that defines $ W $ and it contains $ y $ because so does $ E(x_y, J_y) $ and $ E(x_y, J) \supseteq 
E(x_y, J_y)$ by Corollary \ref{cor: rachability in D([I])}.
\end{proof}

\begin{proposition}\label{prop: G neglig}
$ G:= E\setminus W $ is negligible.
\end{proposition}
\begin{proof}
Let finite sets $ X \subseteq G $ and $ 
Y\subseteq E\setminus G=W $ be given. By Observation \ref{obs: negliDef} we may assume without loss of generality that $ X $ does 
not contain $ N $-loops and $ Y $ does not contain $ M $-loops. Since $ E\setminus 
\mathsf{span}_N(\mathcal{D}) \subseteq W $ and $ G=E \setminus 
W 
$, we have $ 
G 
\subseteq \mathsf{span}_N(\mathcal{D})$. The sets $ X $ and $ Y $ are finite and $ \mathcal{D} $ is a directed set, hence we can pick a single $ 
J \in \bigcup \mathcal{D} $ such that 
$ \mathsf{span}_N(J)\supseteq X $ and $ J \trianglerighteq J_y $ for every $ y \in Y $. Then $ \mathsf{span}_M(J)\supseteq Y $ as well (see 
Lemma \ref{lem: E(x,I) M-span}). Let $ J_X $  be the smallest subset of $ J $ that $ N $-spans $ X $. Formally,  \[ 
J_X:=(J\cap X)\cup  \bigcup\{ C_N(x,J)-x:\ x \in 
X\setminus J 
\}. \]  By 
definition $  
\mathsf{span}_N(J_X) \supseteq X $. It remains to prove that $ J_X $ is $ M/Y $-independent. Let $ J_Y $ be the smallest subset of $ J $ that $ M 
$-spans $ Y $. Formally, 
 \[ J_Y:=(J \cap Y)\cup \bigcup\{ C_M(y,J)-y:\ y \in Y \setminus J \}. \] It is enough to show that $ J_X\cap 
 J_Y=\emptyset $ because then it follows that $ J_X $ is $ M/J_Y $-independent and hence  $ M/Y 
 $-independent. 
 
 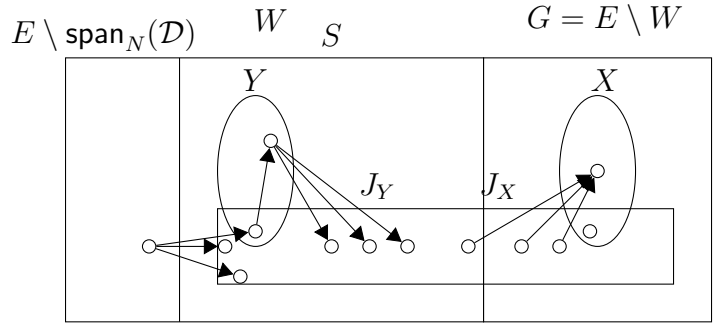
\begin{figure}[h]
 \centering
 
\begin{tikzpicture}

\draw  (-4.5,3.5) rectangle (1,0);
\draw  (1,3.5) rectangle (4,0);
\draw  plot[smooth, tension=.7] coordinates {(-3,3.5) (-3,0)};
\draw  (-2.5,1.5) rectangle (3.5,0.5);
\draw  (-2,2) ellipse (0.5 and 1);
\draw  (2.5,2) ellipse (0.5 and 1);
\node[circle,inner sep=0pt,draw,minimum size=5] (v9) at (1.5,1) {};
\node[circle,inner sep=0pt,draw,minimum size=5] (v11) at (2,1) {};
\node[circle,inner sep=0pt,draw,minimum size=5] (v10) at (2.5,2) {};
\node[circle,inner sep=0pt,draw,minimum size=5] (v6) at (-1,1) {};
\node[circle,inner sep=0pt,draw,minimum size=5] (v7) at (-0.5,1) {};
\node[circle,inner sep=0pt,draw,minimum size=5] (v8) at (0,1) {};
\node[circle,inner sep=0pt,draw,minimum size=5] (v2) at (-2.2,0.6) {};
\node[circle,inner sep=0pt,draw,minimum size=5] (v3) at (-2.4,1) {};
\node[circle,inner sep=0pt,draw,minimum size=5] (v5) at (-1.8,2.4) {};
\node[circle,inner sep=0pt,draw,minimum size=5] (v4) at (-2,1.2) {};
\node[circle,inner sep=0pt,draw,minimum size=5] at (2.4,1.2) {};
\node[circle,inner sep=0pt,draw,minimum size=5] (v1) at (-3.4,1) {};
\node[circle,inner sep=0pt,draw,minimum size=5] (v12) at (0.8,1) {};

\draw[-triangle 60]  (v1) edge (v2);
\draw[-triangle 60]  (v1) edge (v3);
\draw[-triangle 60]  (v1) edge (v4);
\draw[-triangle 60]  (v4) edge (v5);
\draw[-triangle 60]  (v5) edge (v6);
\draw[-triangle 60]  (v5) edge (v7);
\draw[-triangle 60]  (v5) edge (v8);
\draw[-triangle 60]  (v9) edge (v10);
\draw[-triangle 60]  (v11) edge (v10);
\draw[-triangle 60]   (v12) edge (v10);

\node at (-4,3.8) {$E\setminus \mathsf{span}_N(\mathcal{D})$};
\node at (-1.8,4) {$W$};
\node at (2.6,4) {$G=E\setminus W$};
\node at (-0.4,1.8) {$J_Y$};
\node at (1.2,1.8) {$J_X$};
\node at (-2,3.2) {$Y$};
\node at (2.6,3.2) {$X$};
\node at (-1,3.8) {$S$};

\end{tikzpicture}
 \caption{The negligibility of $ G $.} \label{fig: JX and JY}
 \end{figure}
\noindent It follows from 
  the definition of $ D(J) $ that $ X $ is reachable by a directed path from each element of $ J_X $ in $ D(J) $. Indeed, an $ x \in 
 J_X $ is either an element of $ X $ or $ D(J) $ has an arc of second type from $ x $ to $ X $. Similarly, $ J_Y $ 
 is reachable from 
 each element of $ Y $ since any $ y \in 
  Y $ is either an element of $ J_Y $ or $ D(J) $ has an arc of the first type from $ y $ to $ J_Y $. Suppose for a contradiction that $ J_X \cap 
  J_Y\neq 
  \emptyset $.  Then there exists a directed path (of length at most two) from $ Y $ to $ X $ in $ D(J) $. But this is a 
  contradiction  because  $ E(y,J)\subseteq W $ for $ y \in Y $  by Observation 
    \ref{obs: reach from y} while $ 
      X\subseteq G=E\setminus W $.  We conclude $ J_X\cap J_Y=\emptyset $ which completes the proof.
\end{proof}

\begin{observation}\label{obs: L spanned by S}
$ S $ spans $ W $ in $ M $.
\end{observation}
\begin{proof}
Let $ x \in W\setminus S=E\setminus \mathsf{span}_N(\mathcal{D}) $ be given. We may assume that $ x $ is not an $ M $-loop, since otherwise 
we are done. Then $ N^{+}_{D(I_x)}(x)+x $ must be an $ M $-circuit because $ \{ x \} $ cannot be an augmenting path 
for $ I_x $. We have 
$N^{+}_{D(I_x)}(x)\subseteq W $ by the definition of $ W $. Furthermore, $ N^{+}_{D(I_x)}(x)\subseteq I_x \subseteq 
\mathsf{span}_N(\mathcal{D}) $. Therefore $  N^{+}_{D(I_x)}(x)\subseteq W\cap \mathsf{span}_N(\mathcal{D}) 
=S  $.
\end{proof}

\begin{proof}[Proof of Lemma \ref{lem: stable redaction}]
Since there is no non-trivial negligible set, Proposition \ref{prop: G neglig} implies that $ E=W $ and thus $ 
S=\mathsf{span}_N(\mathcal{D}) $. 
Let $ x \in S $ be given and take $ I_0:=J_{x} $.  If $ I_n \in \bigcup \mathcal{D} $ 
is already defined, then by using the fact that $ 
\mathcal{D} $ is a directed set, we take an $ [I_{n+1}] \in 
{[I_n]}{\uparrow_{\mathcal{D}}} $ such that $ [I_{n+1}] \trianglerighteq [J_y] $ for every $ y \in I_n $. By Lemma 
\ref{lem: reach some J'}, $ I_{n+1} $ can be 
chosen to be an element of 
 $ \mathcal{A}(I_n)$. The recursion is done. First, we show that the sequence $ (I_n) $ 
is convergent, i.e.  
\[  \bigcup_{m \in \mathbb{N}} \bigcap_{n>m} I_n=\bigcap_{m \in \mathbb{N}} \bigcup_{n>m} I_n. \]
In other words, the sequence of characteristic functions of the sets $ I_n $ converges pointwise. Indeed, if $ y \in I_n $, then 
$\mathcal{D}\ni [I_{n+1}] 
\trianglerighteq [J_y] $ by construction. Thus $ y \in E(x_y, I_{n+1}) $ and Lemma \ref{lem: E(x, I) increasing} guarantees 
that for $ m \geq n+1 $ we have $ y \in I_m 
$ iff 
$ y \in I_{n+1} $. 

Let $ I $ be the limit of the sequence $ (I_n) $. 
Clearly $ I \in \mathcal{I}_M\cap \mathcal{I}_N $, because $ M $ and $ N $ are finitary.
It follows from the construction via Lemmas \ref{lem: E(x, I) increasing} and \ref{lem: E(x,I) M-span} that $ \bigcup_{n \in \mathbb{N}} I_n 
\subseteq 
\mathsf{span}_M(I) $. In particular, $ x \in \mathsf{span}_M(I) $. 

It remains to prove that $ I $ is $ (M\!\!\upharpoonright\!\!S, N\!\!\upharpoonright\!\! S) $-stable. Suppose for a contradiction that it is 
not and pick a minimal $ J \in    (\mathcal{I}_{M\!\upharpoonright\!S/I}\cap \mathcal{I}_{N\!\upharpoonright\! S})\setminus  
\mathcal{I}_{{N\!\upharpoonright\! S}/I} $. By the minimality, $ I \cup J $ contains a unique $ N $-circuit $ C $ 
and $ J=C\setminus I $ (see Figure \ref{fig: I is stable}). Since $ 
I $ is $ 
N $-independent, $  J\neq \emptyset $, thus we can pick a $ z \in J $.  Fix an $ n $ for which $ I_n \trianglerighteq J_y $ for every $ y 
\in I \cap C $, in particular $C \cap 
I\subseteq I_n $.  Then $ K:= I_n \cup C-z \in \mathcal{I}_M\cap \mathcal{I}_N $ and $ z \in E \setminus \mathsf{span}_M(K) $. 
Furthermore, $ 
yz \in D(K) $ for every $ y \in I \cap C $. Hence for every $ y \in I \cap C $, there exists an $ x_yz $-augmenting path  for $ K $. It is enough to 
show 
that $ K \in \bigcup\mathcal{D} $  because then the existence of such an augmenting path contradicts
the definition of $ I_{x_y} $. It suffices to find a $ [K'] \in \mathcal{D} $ with $ [K']\trianglerighteq [K]  
$ because then Observation \ref{obs: nonempty dowclosed D} ensures $ [K]\in \mathcal{D} $.  The set $ C\setminus I_n 
\subseteq S=\mathsf{span}_N(\mathcal{D}) $ is finite, thus we can pick un upper bound $ [K']\in \mathcal{D} $ for $ [J_y]\ (y\in 
C\setminus I_n) $ and $ [I_n] $  that $ N $-spans $ C\setminus I_n $. Then this $ [K'] $ also $ M $-spans each $ 
y\in  C\setminus I_n $ because $  [K'] \trianglerighteq [J_y] $ (see the definition of $ J_y $ and Lemma \ref{lem: E(x,I) 
M-span}), thus $ [K'] \trianglerighteq [K]$ as desired.

 \begin{figure}[h]
 \centering
 
\begin{tikzpicture}

\draw  (-0.8,0.8) rectangle (4.6,-1);
\draw  plot[smooth, tension=.7] coordinates {(3,1.4) (1.4,1.2) (1.4,0) (2.8,-0.2) (3.6,0.4) (3,1.4)};
\node[circle,inner sep=0pt,draw,minimum size=5] (v1) at (-0.8,1.4) {};
\node[circle,inner sep=0pt,draw,minimum size=5] (v2) at (-0.2,0.2) {};
\node[circle,inner sep=0pt,draw,minimum size=5] (v3) at (0.4,1.4) {};
\node [circle,inner sep=0pt,draw,minimum size=5] (v5) at (3,1.4) {};
\node[circle,inner sep=0pt,draw,minimum size=5] (v4) at (1.4,0) {};
\node[circle,inner sep=0pt,draw,minimum size=5] (v7) at (2.8,-0.2) {};
\node[circle,inner sep=0pt,draw,minimum size=5] (v6) at (2.2,-0.2) {};
\node[circle,inner sep=0pt,draw,minimum size=5] at (1.4,1.2) {};
\node[circle,inner sep=0pt,draw,minimum size=5] at (2.2,1.4) {};

\draw[-triangle 60]  (v1) edge (v2);
\draw[-triangle 60]  (v2) edge (v3);
\draw[-triangle 60]  (v3) edge (v4);
\draw[-triangle 60]  (v4) edge (v5);
\draw[-triangle 60]  (v6) edge (v5);
\draw[-triangle 60]  (v7) edge (v5);

\node at (0.8,-0.6) {$I_n$};
\node at (-0.5,1.4) {$x_y$};
\node at (1.4,-0.3) {$y$};
\node at (3.2,1.4) {$z$};
\node at (1.8,1.6) {$J$};
\node at (3.2,0.2) {$C$};
\end{tikzpicture}
 \caption{The stability of $ I $.} \label{fig: I is stable}
 \end{figure}
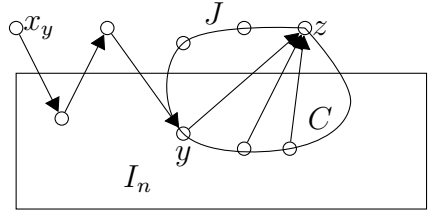
\noindent Therefore $ I $ is indeed $ (M\!\!\upharpoonright\!\!S, N\!\!\upharpoonright\!\! S) $-stable which concludes the proof of 
    Lemma \ref{lem: stable redaction}. Thus the proof of Theorem \ref{t: main} is also complete.
\end{proof}
\printbibliography
\end{document}